\documentclass[final,onefignum,onetabnum]{siamart171218}

\usepackage{fix-cm}
\usepackage{arydshln}
\usepackage{mathtools}
\usepackage{graphicx}
\usepackage{psfrag}
\usepackage{epsfig}
\usepackage{url}
\usepackage{stfloats}
\usepackage{graphics,bm,paralist,theorem,ifthen,color, algorithm}
\usepackage{amsmath}
\usepackage{amsfonts}
\usepackage{amssymb}
\usepackage{array}
\usepackage{subfig}
\usepackage{algpseudocode}
\usepackage{calc}
\usepackage{longtable}
\usepackage{mathrsfs}
\usepackage{float}
\usepackage{lipsum}
\usepackage{epstopdf}

\ifpdf
  \DeclareGraphicsExtensions{.eps,.pdf,.png,.jpg}
\else
  \DeclareGraphicsExtensions{.eps}
\fi

\newsiamremark{remark}{Remark}
\newsiamremark{hypothesis}{Hypothesis}
\crefname{hypothesis}{Hypothesis}{Hypotheses}
\newsiamthm{claim}{Claim}

\headers{}{Maher Nouiehed, Jason Lee, and Meisam Razaviyayn}

\title{Convergence to Second-Order Stationarity for Constrained Non-Convex Optimization}

\author{Maher  Nouiehed\thanks{
Department of Industrial Engineering and Management, American University of Beirut, Beirut, Lebanon 
  (\email{mn102@aub.edu.lb}).}
\and Jason Lee\thanks{Department of Electrical Engineering, Princeton University, Princeton, NJ
  (\email{jasonlee@princeton.edu}).}
  \and Meisam Razaviyayn\thanks{Department of Industrial and Systems Engineering, University of Southern California, Los Angeles, CA 
  (\email{razaviya@usc.edu}).} }

\usepackage{amsopn}

\makeatletter
\def\BState{\State\hskip-\ALG@thistlm}
\makeatother

\algnewcommand{\IfThenElse}[3]{
	\State \algorithmicif\ #1\ \algorithmicthen\ #2\ \algorithmicelse\ #3}

\DeclareMathOperator*{\argmin}{arg\,min}
\newcommand{\bzero}{\mathbf{0}}
\newcommand{\st}{{\rm s.t.}}
\newcommand{\cX}{{\cal X}}
\newcommand{\cpsi}{{\cal \psi}}
\newcommand{\bx}{\mathbf{x}}
\newcommand{\bbx}{\bar{\mathbf{x}}}
\newcommand{\btx}{\widetilde{\mathbf{x}}}

\newcommand{\bQ}{{\mathbf{Q}}}
\newcommand{\bA}{{\mathbf{A}}}
\newcommand{\bI}{{\mathbf{I}}}
\newcommand{\bones}{{\mathbf{1}}}

\newcommand{\bs}{\mathbf{s}}
\newcommand{\bd}{\mathbf{d}}

\newcommand{\bv}{\mathbf{v}}
\newcommand{\ball}{{\rm I\!B}}
\algdef{SE}[LOOP]{Loop}{EndLoop}[1]{\algorithmicloop\ #1}{\algorithmicend\ \algorithmicloop}

\newtheorem{assumption}[theorem]{Assumption}

\ifpdf
\hypersetup{
	pdftitle={An Example Article},
	pdfauthor={Maher Nouiehed, Jason D. Lee, and Meisam Razaviyayn}
}
\fi

\begin{document}
	\maketitle
	
	\begin{abstract}
		We consider the problem of finding an approximate second-order stationary point of a constrained non-convex optimization problem. 
		We first show that, unlike the gradient descent method for unconstrained optimization, the vanilla projected gradient descent algorithm may converge to a strict saddle point even when there is only a single linear  constraint. We then provide a hardness result by showing that checking $(\epsilon_g,\epsilon_H)$-second order stationarity is NP-hard even in the presence of linear constraints. Despite our hardness result, we identify instances of the problem for which checking  second order stationarity can be done efficiently. For such instances, we propose a dynamic second order Frank--Wolfe algorithm which converges to ($\epsilon_g, \epsilon_H$)-second order stationary points in ${\mathcal{O}}(\max\{\epsilon_g^{-2}, \epsilon_H^{-3}\})$ iterations. The proposed algorithm can be used in general constrained non-convex optimization as long as the constrained quadratic sub-problem can be solved efficiently.  
	\end{abstract}
	
	\section{Introduction}
	
	Designing efficient algorithms for non-convex optimization has recently been an active area of research, see \cite{anandkumar2016efficient, cartis2012adaptive, cartis2011adaptive, cartis2011adaptive-2, cartis2013evaluation, cartis2015evaluation,  cartis2017second, curtis2017inexact,curtis2017trust,   lu2012trust,  nesterov2006cubic}. For a general non-convex problem, even finding a local optimum of the objective function is in general NP-Hard \cite{murty1987some}. Therefore, in practice, most existing algorithms converge to first or second order stationary points. The latter provides stronger guarantees as it constitutes a smaller subset of the critical points of the objective function that includes local and global optima. Moreover, when applied to functions with ``nice'' geometrical properties, the set of second order stationary points could even be the same as the set of global optima; see \cite{bandeira2016low, barazandeh2018behavior, boumal2016nonconvex, nouiehed2018learning, sun2016geometric, sun2017complete-2,  sun2017complete} for examples of such objective functions. \\

	Most existing algorithms used for finding second order stationary points focus on  \textit{unconstrained} optimization problems. As a first order algorithm, \cite{ge2015escaping} proposes a noisy-stochastic gradient descent method and shows convergence to the set of second-order stationary points, almost surely. Similarly, \cite{lee2016gradient} uses stable manifold theorem to show that gradient descent with random initialization and sufficiently small constant step size converges to the latter set of second-order stationary points. More specifically, they show that the set of initial points that converge to a \textit{strict saddle point} of the objective function is a measure zero set. Under the assumption that \textit{strict saddle property} holds, the latter methods converge to a local optimum of the objective function. As a negative result, \cite{du2017gradient} constructs an example where the simple gradient descent can take exponential number of iterations to escape a strict saddle point. Motivated by this example, \cite{jin2017escape} proposes a perturbed form of gradient descent that can efficiently escape saddle points under strict saddle property.\\

	
	
	Using higher order derivative information of the objective function, \cite{cartis2011adaptive, cartis2011adaptive-2, curtis2017trust, nesterov2006cubic,royer2018complexity} propose trust region or cubic regularization methods for finding second order stationary points for unconstrained optimization problems. More specifically, the traditional trust region method \cite[Algorithm~6.1.1]{conn2000trust} and cubic regularization methods, which are based on the work of \cite{griewank1981modification, nesterov2006cubic}, are able to converge to second order stationary points. Moreover,  \cite{cartis2011adaptive, cartis2011adaptive-2} proposed the Adaptive Regularization framework using Cubics, also known as (ARC), and established convergence to a second order stationary point. This method computes at each iteration an (approximate) global optimum for a local cubic model which resembles the behavior of the original objective function. They show that  ARC  requires ${\cal O}(\epsilon^{-3/2})$ iterations to converge to an $\epsilon$-first-order stationary point, and ${\cal O}(\epsilon^{-3})$ iterations to reach an $\epsilon$-second-order stationary point. Motivated by these rates, \cite{curtis2017trust} designed a trust region method, entitled TRACE, which has the same iteration complexity bound as  ARC. TRACE  alters the acceptance criteria adopted in the traditional trust region method, and introduces a new mechanism for updating the radius of the trust region. In a more recent work, \cite{curtis2017exploiting} developed an algorithm with a dynamic choice for direction and step-size. In particular, the dynamic algorithm decides on the step that offers a more significant predicted reduction in the objective value. The dynamic algorithm, ARC and TRACE methods fall under a more general framework that was proposed in \cite{curtis2017inexact}. This framework uses a set of generic conditions that need to be satisfied at each trial step, and converges to second order stationarity with optimal iteration complexity bound.\\
	
	For constrained optimization problems, many recent papers propose algorithms that converge to first-order and second-order stationary points. For example, \cite{lacoste2016convergence} proposed a Frank--Wolfe algorithm that converges to an $\epsilon$-first order stationary point with complexity ${\cal O}(\epsilon^{-2})$. Another work by \cite{ghadimi2016mini} shows that projected gradient descent converges to an $\epsilon$-first order stationary point with the same complexity bound. Convergence to second-order stationary points can be achieved by extending some of the aforementioned second or third order methods. For instance, \cite{cartis2012adaptive} adapted the ARC method and showed that the worst-case function evaluation complexity for converging to an $\epsilon$-first order stationary point is ${\cal O}(\epsilon^{-3/2})$. Moreover, \cite{cartis2015evaluation} showed that the same rate of convergence can be achieved for solving general smooth problems involving both non-convex equality and inequality constraints, using a cubic regularization method. In addition, a \textit{conceptual} trust region method was proposed in \cite{cartis2017second} to compute an $\epsilon$-approximate $q$-th order stationary point in at most ${\cal O}(\epsilon^{-q-1})$ iterations. The iteration complexity bounds computed for these methods \textit{hide the per-iteration complexity of solving the sub-problem}. These sub-problems are either quadratic or cubic constrained optimization problems, which are in general NP-complete; see section~\ref{sec:NPhard} in this paper. \\
	
	{\color{black} Concurrent to this work, \cite{mokhtari2018escaping} proposed a general framework that yields convergence to an approximate $(\epsilon_g, \epsilon_H)$-second order stationary point in at most ${\cal O}(\max\{\epsilon_g^{-2}, \epsilon_H^{-3}\})$ iterations. This is achieved when the feasible set is convex and compact. In particular, the framework uses a first order method to converge to an approximate first order stationary point, and then computes a second order descent direction if it exists. However, as we show in Section~\ref{Section-Stationarity-Def}, given any positive scalars $\epsilon_g$ and $\epsilon_H$, the algorithm may converge to an approximate strict saddle point. More specifically, we provide a counterexample for which given any positive scalars $(\epsilon_g, \epsilon_H)$ applying the algorithm proposed in \cite{mokhtari2018escaping} finds an approximate strict saddle point. Since solving the quadratic sub-problem to optimality is NP-Hard, they suggest to approximately solve the sub-problems. In this paper, we show that, even for linear constraints, finding an approximate solution for these sub-problems is NP-Hard.}\\
	
	In addition to these second order methods, first order methods have also been used for finding second order stationary points of optimization problems with \textit{manifold constraints}. The recent work \cite{lee2017first} shows that manifold gradient descent converges to local minima under the strict saddle property. More recently, \cite{hong2018gradient} established a similar result for a primal-dual optimization procedure implemented to solve \textit{linear equality constrained} optimization problems. Unlike the linear equality constrained scheme, the convergence of first order algorithms to second-order stationarity is poorly understood in the presence of linear \textit{inequality} constraints.\\ 
	
	In this paper, we first provide an example that shows that projected gradient descent algorithm may converge to strict saddle points with positive probability even in the presence of a single linear constraint. We then discuss an NP-hardness result about solving the sub-problem of current second-order methods applied to constrained optimization problems. 	In addition, we provide a counterexample showing that current iterative second-order algorithms may get stuck near a strict saddle point. Motivated by this counterexample and inspired by the algorithms proposed in \cite{curtis2017exploiting} and \cite{lacoste2016convergence}, we propose a  second-order Frank--Wolfe algorithm that uses a dynamic choice for the direction and step-size. Moreover, our algorithm adopts a dynamic mechanism for escaping the region near strict saddle points that can be of independent interest. We show the algorithm converges to ($\epsilon_g, \epsilon_H)$-second order stationary points in ${\cal O}(\max\{\epsilon_g^{-2}, \epsilon_H^{-3}\})$ iterations. Unlike the algorithms proposed in \cite{lacoste2016convergence, mokhtari2018escaping}, our algorithm does not require any boundedness assumption on the feasible set. \\
	%
	%
	
	\section{First and Second Order Stationarity}\label{Section-Stationarity-Def}
	To understand the definitions of first and second order stationarity, we first consider the  unconstrained optimization problem
	\begin{equation}\label{eq:General-Optimization-Prob}
	\underset{\bx \in \mathbb{R}^{n}}{\min} \, \, f(\bx),
	\end{equation}
	where $f: \mathbb{R}^n \mapsto \mathbb{R}$ is a twice continuously differentiable function. We say a point $\bar\bx$ is a first order stationary point of \eqref{eq:General-Optimization-Prob} if $\nabla f(\bar\bx) = \bzero$. Similarly, a point $\bar\bx$ is said to be a second-order stationary point if $\nabla f(\bar\bx) = \bzero$ and $\nabla^2 f(\bar\bx) \succeq 0$. Moreover, we say $\bx$ is a strict saddle point if it is a first-order stationary point but not a second-order stationary  point. If all second order stationary points of the objective function are local optima, we say the function satisfies the \textit{strict saddle} property. This property is satisfied in many practical objective functions; see \cite{ge2015escaping,  kawaguchi2016deep, sun2016geometric, sun2017complete-2, sun2017complete}. In addition, if every local optima of the objective function is globally optimal, then finding the global optimum of the objective boils down to finding a second order stationary point; see \cite{ge2016matrix, lu2017depth, nouiehed2018learning, sun2015nonconvex} for examples of such functions.
	
	\vspace{0.2cm}

	We now extend these definitions to the constrained optimization problem
	\begin{equation}\label{eq:General-Cons-Optimization-Prob}
	\underset{\bx \in \mathcal{F}}{\min} \, \, f(\bx),
	\end{equation}
	where $\mathcal{F}\subseteq \mathbb{R}^n$ is a closed convex set. We say a point $\bar{\bx} \in {\cal F}$ is a first order stationary point of \eqref{eq:General-Cons-Optimization-Prob} if $\langle \nabla f(\bar{\bx}), \bx - \bar{\bx} \rangle \geq 0$ for all $\bx \in \mathcal{F}$, or equivalently if 
	\begin{equation}\label{eq:FOS}
	\begin{split}
	0 =  \min_{\bs}\quad &\langle \nabla f(\bar\bx), \bs \rangle  \\
	\st \quad & \bar\bx + \bs \in {\cal F}, \, \|\bs\|\leq 1.
	\end{split}
	\end{equation}
	Similarly, as defined in \cite{bertsekas1999nonlinear}, we say a point $\bar{\bx} \in {\cal F}$ is a second order stationary point of \eqref{eq:General-Cons-Optimization-Prob} if $\bar{\bx}$ is a first order stationary point and
	\begin{equation}\label{eq:SoSConstrained}
	0\leq \bd^T \nabla^2f(\bar{\bx}) \bd,\quad \forall \, \bd \;\st\; \langle\bd,\nabla f(\bar{\bx}) \rangle =0\textrm{ and } \bar{\bx} + \bd \in \mathcal{F}. 
	\end{equation}
	
	Moreover, we say that (\ref{eq:General-Cons-Optimization-Prob}) satisfies the strict saddle property if every saddle point of the objective is not a second order stationary point. Notice that these definitions simplify to the corresponding unconstrained  definitions when $\mathcal{F} = \mathbb{R}^n$.
	
	\vspace{0.2cm}
	
	Motivated by (\ref{eq:FOS}) and (\ref{eq:SoSConstrained}), given a feasible point $\bx \in {\cal F}$, we define the following first and second order stationarity measures
	\begin{equation}\label{eq:X_k}
	\cX(\bx) \triangleq  -\min_{\bs}\;\;\langle \nabla f(\bx), \bs \rangle \quad \st \quad  \bx + \bs \in {\cal F}, \, \|\bs\|\leq 1.
	\end{equation}
	and 
	\begin{equation}\label{eq:psi_k}
	{\cal \psi}(\bx, \alpha) \triangleq -\min_{\bd}\;\; \bd^T \nabla^2f(\bx)\bd \quad  \st \quad  \bx + \bd \in {\cal F},  \; \|\bd\|\leq 1, \; \langle \nabla f(\bx), \bd \rangle \leq \alpha.
	\end{equation}
	These optimality measures have been previously used in the literature  \cite{conn2000trust, cartis2012adaptive, cartis2015evaluation, nouiehed2019trust}. Notice that since $x$ is feasible, $\cX(\bx) \geq 0$ and $\cpsi (\bx, \alpha) \geq 0$. These optimality measures can be linked to the standard definitions in~\cite{bertsekas1999nonlinear} by the following Lemma.

	\begin{lemma} \label{lem:StationarityContinuous}
		If $\bar\bx \in {\cal F}$ then
		\begin{itemize}
			\item $\cX (\bar\bx) = 0$ if and only if $\bar\bx$ is a first order stationary point.
			\item $\cX(\bar{\bx}) = \cpsi(\bar{\bx}, 0 ) = 0$ if and only if $\bar\bx$ is a second order stationary point.
		\end{itemize}
	\end{lemma}
	
	
	\vspace{0.2cm}

	The above lemma motivates the following definition.
	
	\begin{definition}\label{def:EpsFOSSOS}
		Approximate first and second order stationary points:
		\begin{itemize}
			\item Given a positive scalar $\epsilon_g$, a point $\bar\bx$ is said to be an $\epsilon_g$-first order stationary point of the optimization problem~\eqref{eq:General-Cons-Optimization-Prob} if $\bar\bx \in {\cal F}$ and $\cX(\bar\bx) \leq \epsilon_g$.
			\item Given positive scalars $\epsilon_g$ and $\epsilon_H$, a point $\bar\bx$ is said to be an $(\epsilon_g, \epsilon_H)$-second order stationary point of the optimization problem~\eqref{eq:General-Cons-Optimization-Prob} if $\bar\bx \in {\cal F}$, $\cX(\bar\bx) \leq \epsilon_g$ and $\cpsi(\bar{\bx},0) \leq \epsilon_H$.
		\end{itemize} 
	\end{definition}
	
	\vspace{0.2cm}
	
	For unconstrained optimization problems, $\epsilon_g$-first order stationarity condition defined in~\ref{def:EpsFOSSOS} boils down to the classical condition $\|\nabla f(\bar\bx) \| \leq \epsilon_g$. Similarly,  $(\epsilon_g,\epsilon_H)$-second order stationarity condition defined in~\ref{def:EpsFOSSOS} is equivalent to $\|\nabla f(\bar\bx)\| \leq \epsilon_g$ and $\lambda_{\min}(\nabla f(\bar\bx)) \geq - \epsilon_H$. These are the standard definitions of approximate first and second order stationarity for unconstrained optimization \cite{curtis2017exploiting, curtis2017trust, curtis2017inexact, nesterov2006cubic, cartis2011adaptive-2}.

	\vspace{0.2cm}
	
	\begin{remark}
		As indicated in~\cite{nouiehed2019trust}, there are two major differences between our definition of $(\epsilon_g,\epsilon_H)$-second order stationarity and the definition used in \cite{mokhtari2018escaping}. Firstly, the definition used for approximate first and second order stationarity in~\cite{mokhtari2018escaping} does not include the normalization constraints $\|\bs\|\leq 1$ and $\|\bd\| \leq 1$ in~\eqref{eq:X_k} and~\eqref{eq:psi_k}. Secondly,~\cite{mokhtari2018escaping} uses the equality constraint  $\langle \nabla f(\bx), \bd \rangle = 0$ in~\eqref{eq:psi_k} instead of the inequality constraint $\langle \nabla f(\bx), \bd \rangle \leq 0$. The significance of including the normalization constraints and using the inequality constraint $\langle \nabla f(\bx), \bd \rangle \leq 0$ are illustrated in~\cite{nouiehed2019trust}. 
	\end{remark}
	
	
	\begin{remark}[Continuity of $\cpsi(\cdot, 0)$]
	
		While the continuity of $\cX(\cdot)$ was established in~\cite{cartis2012adaptive}, to our knowledge, the question of whether $\cpsi(\cdot, 0)$ is continuous remains unanswered. As a negative result, we provide a counterexample showing that $\cpsi(\cdot, 0)$ is non-continuous. Consider the optimization problem 
		\begin{equation}\label{CounterexampleContinuityOfPsi}
		 \min_{\bx} f(\bx) \triangleq x_1^2 + x_2^2 - 2x_3^2 + x_1 + 0.5x_2x_3 \quad \mbox{s.t.} \quad x_1 \geq 0, \quad -1 \leq x_2, x_3 \leq 0,
		 \end{equation}
		and the sequence of point 
		\[\bx_k = \bar{\bx} + \dfrac{1}{k}\bd, \]
		where  $\bar{\bx} = (0,0,0)$ and $\bd = (0,-1, 0)$. Then,
		\[\begin{array}{ll}
		\cX(\bx_k) & = -\displaystyle{
			\operatornamewithlimits{\min}_{\|\bs\|_2 \leq 1} \;\; s_1 -\dfrac{2s_2}{k}  - \dfrac{s_3}{2k} \quad \mbox{s.t.} \quad s_1 \geq 0, \;\;   -1 + \dfrac{1}{k}\leq s_2 \leq \dfrac{1}{k}, \;\; -1 \leq s_3 \leq 0}\\
		\\
		& = \dfrac{2}{k^2}.
		\end{array}\] 
		Hence, $\cX(\bx_k) \rightarrow 0$ as $k \rightarrow \infty$. Also,
		\[\cpsi(\bx_k, 0) = \min_{\|\bd\|_2 \leq 1}  2d_1^2 + 2d_2^2 - 4d_3^2 + d_2d_3 \quad \mbox{s.t.} \quad \bd \in {\cal D}_k,\] 
		where 
		\[{\cal D}_k \triangleq \left\{ \bd \;\;\Big|\;\; d_1 \geq 0, \quad   -1 + \dfrac{1}{k}\leq d_2 \leq \dfrac{1}{k}, \quad -1 \leq d_3 \leq 0 , \quad d_1 - \dfrac{2d_2}{k} -\dfrac{d_3}{2k} \leq 0 \right\}.\]
		
		As $k \rightarrow \infty$, one can easily check that the feasible set ${\cal D}_k \rightarrow \bzero$. Hence, $\cpsi(\bx^k, 0) \rightarrow 0$. Moreover, 
		
		\[\begin{array}{ll}
		\cpsi(\bar{\bx}, 0) & = \displaystyle{
			\operatornamewithlimits{\min_{\|\bd\|_2 \leq 1}}}  2d_2^2 - 4d_3^2 + d_2d_3 \quad \mbox{s.t.} \quad -1 \leq d_2, d_3 \leq 0\\
		\\
		& = 4.
		\end{array}\] 
		
		We have constructed a counterexample of a sequence of points ${\bx^k} \rightarrow \bzero$, with $\cpsi(\bx^k) \rightarrow 0 \neq \cpsi(\bzero)$. Therefore, $\cpsi(\cdot)$ is a non-continuous function. According to our definition~\ref{def:EpsFOSSOS}, given any prescribed positive constants $\epsilon_g$ and $\epsilon_H$ and for sufficiently large $k$, $\bx_k$  defined in our counterexample is an $(\epsilon_g, \epsilon_H)$-second order stationary point. However, $\bx_k$ can be arbitrarily close to $\bar{\bx}$ which is a first order stationary point but not a second-order stationary point. The former undesirable scenario can occur in practice when using the algorithm proposed in~\cite{mokhtari2018escaping}. More specifically, applying the algorithm  proposed in~\cite{mokhtari2018escaping} to minimize the function~\eqref{CounterexampleContinuityOfPsi} with initial point $\bx_0 = (0,-1/2, 0)$ finds a point near the origin which is a first but not second order stationary point. To avoid the former undesirable scenario, we have to show that if the sequence of iterates computed by the algorithm $\bx_k \rightarrow \bar{\bx}$ and  $\cpsi(\bx_k) \rightarrow 0$, then $\cpsi(\bar{\bx}) = 0$. We propose an algorithm with this former property in Section~\ref{Section_FW}.
		\end{remark}

	In the unconstrained scenario, it is well-known that gradient descent with random initialization converges to second order stationary points with probability one \cite{lee2016gradient}. Moreover, there exist various efficient algorithms for finding an $(\epsilon_g,\epsilon_H)$-second order stationary point of the objective function \cite{nesterov2006cubic, curtis2017trust, curtis2017exploiting, curtis2017inexact,  cartis2011adaptive, cartis2011adaptive-2}. In what follows, we study whether these results can be directly extended to the constrained scenario by answering the following questions:
	
	\vspace{0.5cm}
	
	\begin{center}
		\noindent\fbox{
			\parbox{0.9\textwidth}{
				\begin{itemize}
					\item[\textbf{Question 1:}] Does projected gradient descent with random initialization converge to second order stationary points with probability one?
					\vspace{0.3cm}
					\item[\textbf{Question 2:}] Does there exist an efficient algorithm for finding an $(\epsilon_g,\epsilon_H)$-second order stationary point of the general constrained optimization problem \eqref{eq:General-Cons-Optimization-Prob}?
				\end{itemize}
			}
		}
	\end{center}

	\vspace{0.7cm}

	\section{Projected Gradient Descent with Random Initialization May Converge to Strict Saddle Points with Positive Probability}
	\label{sec:CounterExample}
	It is known that gradient descent with fixed step size can converge to an $\epsilon$-first order stationary point in $\mathcal{O}(\epsilon^{-2})$ iterations for unconstrained smooth optimization problems \cite{nesterov2013introductory}. Moreover, with random initialization, this method escapes strict saddle points for general smooth unconstrained optimization problems with probability one  \cite{lee2016gradient}. 
	In the general constrained optimization problem \eqref{eq:General-Cons-Optimization-Prob}, projected gradient descent algorithm is a natural replacement for gradient descent. The iterates of the projected gradient descent algorithm are obtained by
	\[
	x_{k+1}   \leftarrow \mathcal{P}_{\mathcal{F}} \left( x_k - \alpha_k \nabla f(x_k)\right),
	\]
	where $\alpha_k$ is the step-size, $k$ is the iteration number, and $\mathcal{P}_{\mathcal{F}}$ is the projection operator onto the feasible set $\mathcal{F}$. 
	A natural question about projected gradient descent is whether it has the same behavior as the 
	
	gradient descent algorithm. More specifically, can projected gradient descent escape strict saddle points?
	To answer this question, we provide an example where projected gradient descent fails to converge to second order stationary points even in the presence of a single linear constraint.
	
	\vspace{0.2cm}
	
	Consider the following optimization problem:
	\begin{equation}\label{P1}
	\underset{x, \, y \in \mathbb{R}}{\min} \, \, f(x,y) \triangleq -x y e^{-x^2 - y^2} + \dfrac{1}{2} y^2  \quad \mbox{s.t.} \quad x+y \leq 0.
	\end{equation}
	The landscape of the function $f$ and its corresponding negative gradient mapping are plotted in Figures \ref{fig1:a} and \ref{fig1:b}. Notice that the function $f(\cdot)$ has the following first and second order derivatives: 
	\[
	\nabla f(x,y) = \left(
	\begin{array}{c}
	-(1-2x^2)y e^{-x^2 - y^2} \\
	-(1-2y^2)x e^{-x^2 - y^2} + y
	\end{array}
	\right)
	\]
	\[\nabla^2 f(x,y) = \begin{pmatrix}
	2xy (3-2x^2)e^{-x^2 - y^2} & -(1-2x^2)(1-2y^2)e^{-x^2-y^2}\\
	-(1-2x^2)(1-2y^2)e^{-x^2-y^2} & 2xy(3-2y^2)e^{-x^2 - y^2} +1
	\end{pmatrix}.
	\]

	\begin{figure}[h]
		\centering 
		\subfloat[Landscape of $f(\cdot)$]{\label{fig1:a}\includegraphics[scale = 0.45]{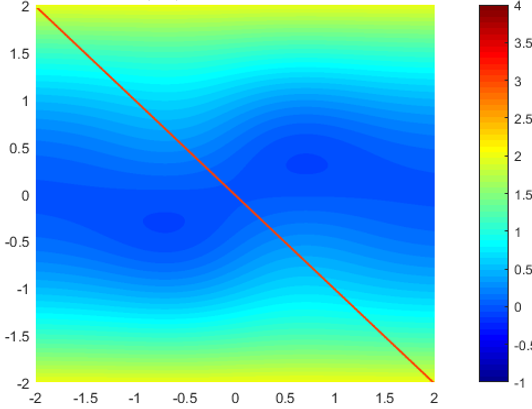}}
		\hspace{0.7cm} 
		\subfloat[Negative Gradient Flow for $f(\cdot)$]{\label{fig1:b}\includegraphics[scale = 0.45]{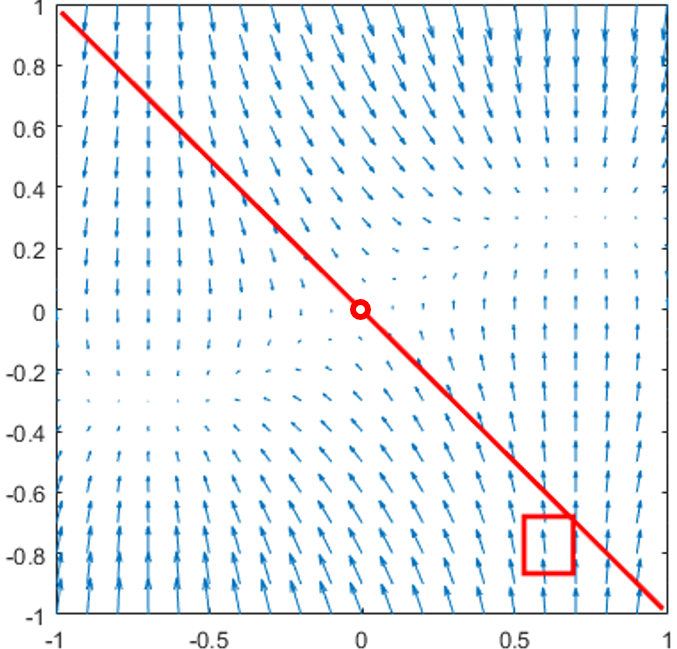}}
		\caption[Landscape and negative gradient mapping of $f$]{The landscape and negative gradient mapping of $f$. The red box in \ref{fig1:b} shows a non-zero measure set that converges to the origin when projected gradient descent is used.}
	\end{figure}
	
	First of all, it is not hard to check that  $\nabla f(0,0) = \bzero$, $\nabla^2 f(0,0) = \begin{pmatrix} 0 & -1\\  -1 & 1 \end{pmatrix}$, and for the feasible direction $\bv=(-1,-1)$, we have $\bv^T\nabla^2f(0,0) \bv = -1$. Hence, the point $(0,0)$ is a saddle point that is not second order stationary. Therefore, the origin is a strict saddle point.  However, as one can see in Figure \ref{fig1:b}, projected gradient descent algorithm may converge to the origin if initialized around the lower right corner of the figure. This observation is true for various step-size selection rules. 
	To formalize this observation, in what follows, we show that projected gradient descent converges to the strict saddle point $(0,0)$ if initialized inside the red box in Figure~\ref{fig1:b}.
	
	\vspace{0.1cm}
	
	First, we show that if the sequence generated by projected gradient descent method intersects a subset of the boundary of the constraint in (\ref{P1}), then the algorithm will eventually converge to the origin.
	
	\begin{lemma}\label{lem:conditions-converge-to-origin}
		If for any $k \in \mathbb{N}_{+}$, the iterate $(x_k, y_k)$ of the sequence generated by projected gradient descent method with constant step-size $\alpha_k =\bar{\alpha}$ with $0 <\bar{\alpha} < 2/3$ applied to~\eqref{P1} satisfies
		\begin{equation}\label{eq:conditions-to-converge-to-origin}
		x_k \geq 0, \quad y_k = -x_k,
		\end{equation}
		then $\{(x_k,y_k)\}$ converges to the origin.
	\end{lemma}
	
	\begin{proof}
		Proof of this lemma is relegated to Appendix~\ref{Apx:lem:StayonLine}.
	\end{proof}
	
	It remains to show that there exists a non-zero measure region so that if we initialize the projected gradient descent algorithm in this region, the  iterates converge to a point on the boundary satisfying the conditions in Lemma~\ref{lem:conditions-converge-to-origin}.
	
	\begin{theorem}\label{thm:non-zero-measure-set}
		For any given constant step-size $\alpha_k = \bar\alpha$  with $0 < \bar{\alpha} < \dfrac{2}{3}$, there exists $\epsilon > 0$ so that if we initialize in the set
		\[ \ball_{\epsilon} \triangleq \{(x,y) \, | \, 0.5 -\epsilon \leq  x \leq 0.5, \quad -0.5-\epsilon \leq y \leq -0.5\},\]
		then the projected gradient descent method with fixed step-size $\bar{\alpha}$ converges to the origin when applied to \eqref{P1}.
	\end{theorem} 
	
	\begin{proof}
		Proof of this Theorem is relegated to Appendix~\ref{Apx:lem:boxtoline}.
	\end{proof}
	
	This result shows that there is a positive probability that projected gradient descent with random initialization converges to a strict saddle point of the objective.  Based on our example, we conjecture that even perturbed/stochastic projected gradient descent algorithm cannot help in escaping strict saddle points. Therefore, a natural question to ask is whether there exists a polynomial-time algorithm for finding an (approximate) second order stationary point. This question is the focus of the next section. 
	
	\vspace{0.2cm}
	
	\section{Finding or Checking (Approximate) Second Order Stationarity is NP-Hard Even in the Presence of Linear Constraints} \label{sec:NPhard}
	
	The classical result of \cite{murty1987some} shows that checking whether a point $\bd$ is a second-order stationary point of a non-convex quadratic optimization problem is in general NP-hard. This results was shown by considering the quadratic co-positivity problem
	\begin{equation}\label{co-positivity-problem}
	\min_{\bd \in \mathbb{R}^n} \quad \frac{1}{2} \bd^T \mathbf{Q}\bd \quad \quad \st \quad  \bd \geq \bzero.
	\end{equation}
	In particular, \cite[Lemma~2]{murty1987some} shows that, by adding a ball constraint $\|\bd\|_2\leq 1$, the optimal objective value of (\ref{co-positivity-problem}) is either $0$ or $\leq -2^{-n}$. Thus, checking whether $\bar\bd = 0$ is an $(\epsilon_g,\epsilon_H)$-second order stationary point requires choosing $\epsilon_H \leq -2^{-n}$. In that case \cite[Theorem~2]{murty1987some} shows that the problem is NP-Hard. Combining these two results, we conclude that checking whether $\bar\bd = 0$ is an $(\epsilon_g,\epsilon_H)$-second order stationary point is NP-hard in $(n, \log\left(1/\epsilon_H\right))$. In this section, we show that even a less ambitious goal is NP-hard. More precisely, we show that checking $(\epsilon_g,\epsilon_H)$-second order stationarity is NP-hard in $(n, 1/\epsilon_H)$. 
	
	Before proceeding to the result, let us define some notations. Let $G(V,E)$  be a graph with the set of vertices $V$ and the set of edges $E$. Also let $|V|$  be the cardinality of $V$ and  $A_G$  be the adjacency matrix of graph $G$. We define ${\cal C}_n \triangleq \{\bQ \in \mathbb{R}^{n \times n} \, | \, \bx^T\bQ\bx \geq 0, \, \, \forall \, \, \bx \geq \bzero\}$ to be the set of co-positive matrices. We denote the  all-one matrix of size $n$ by $\bones_n$. We say graph $G$ has a stable set of size $t$ if it contains a subset of $t$ vertices, from which no two vertices in the subset are connected by an edge.
	
	\begin{lemma}\label{stable-min-equiv}
		Let $G=(V,E)$ be a graph with $|V|=n$. Given a scalar $t$ with $t \leq n$, define
		\[\bQ = (\bI_n + \bA_G)(t-\dfrac{1}{2}) - \bones_n, \quad \mbox{and} \quad \delta =\dfrac{1}{2n+1}.\]
		Then the following are equivalent:\\
		\begin{enumerate}[i.]
			\item $\underset{\bd \geq \bzero, \, \|\bd\|_2 \leq 1}{\min}\,\, \bd^T\bQ\bd \leq -\frac{\delta}{\sqrt{n}}.$\\
			\item $G$ contains a stable set of size $t$
		\end{enumerate}
	\end{lemma}
	
	\begin{proof}
		We first show that $i$ implies $ii$. By the definition of the set ${\cal C}_n$, the condition
		\[
		\min_{\|\bd\|_2 \leq 1, \,\; \bd \geq \bzero}\,\, \bd^T\bQ\bd \leq -\frac{\delta}{\sqrt{n}}
		\]
		implies that  $\bQ \notin {\cal C}_n$. Therefore, by \cite[Lemma~4.1]{dickinson2014computational},  $G$ contains a stable set of size $t$.\\
		
		To show the converse, we use \cite[lemma~4.5]{dickinson2014computational}. Suppose that $G$ contains a stable set of size $t$. By \cite[lemma~4.1]{dickinson2014computational}, $\bQ \notin \mathcal{C}_n$. Moreover, by \cite[lemma~4.5]{dickinson2014computational} it is $\delta$ far away from~$\mathcal{C}_n$. In other words, there exists a $\delta > 0$ such that
		\begin{equation} \label{eq:appx:Copositive}
		\|\mathbf{Y} - \bQ\|_F > \delta, \quad \forall\, \mathbf{Y}\in \mathcal{C}_n.
		\end{equation}
		
		Let $\displaystyle{\bar\bd \in \argmin_{\bd\geq \bzero,\;\;\|\bd\|_2\leq 1}}\bd^T \bQ \bd$. If $\|\bar\bd\|_2 <1$, we can clearly scale $\bar\bx$ and reduce the objective value. Hence, $\|\bar\bd\|_2 = 1$. We now define,  
		\[\widetilde{\bQ} =\bQ + \frac{\delta}{\sqrt{n}}\mathbf{I}_n.\]
		By~\eqref{eq:appx:Copositive}, we get that $\widetilde{\bQ} \notin \mathcal{C}_n$. If 
		\begin{equation}\label{eq:appx:Q_tilde_opt}
		\displaystyle{\bar\bd \in \argmin_{\bd\geq \bzero,\;\;\|\bd\|_2\leq 1}}\bd^T \widetilde{\bQ} \bd,
		\end{equation}
		then we have,
		\begin{align}
		\min_{\bd\geq \bzero,\;\; \|\bd\|_2\leq 1} \bd^T \bQ \bd
		&= \bar\bd^T \bQ \bar\bd  \nonumber\\
		&= \bar\bd^T \widetilde{\bQ} \bar\bd - \frac{\delta}{\sqrt{n}}\|\bar\bd\|^2\nonumber\\
		&= \min_{\bd\geq \bzero,\;\;\|\bd\|_2\leq 1} \bd^T \widetilde{\bQ} \bd - \frac{\delta}{\sqrt{n}} \nonumber\\
		&\leq -\frac{\delta}{\sqrt{n}} \nonumber
		\end{align}
		where the third equality holds by~\eqref{eq:appx:Q_tilde_opt}. This directly implies $ii$.\\
		
		To complete the proof, we still need to show that 
		\[\displaystyle{\bar\bd \in \argmin_{\bd\geq \bzero,\;\;\|\bd\|_2\leq 1}}\bd^T \widetilde{\bQ} \bd.\]
		Assume the contrary, then we can define $\widetilde{\bd} \neq \bar\bd$ to be the minimizer
		\[\displaystyle{\widetilde{\bd} \in \argmin_{\bd\geq \bzero,\;\;\|\bd\|\leq 1}}\bd^T \widetilde{\bQ} \bd.\]
		By the same previous argument, $\|\widetilde{\bd}\|_2 = 1$. Hence, 
		\[\begin{array}{ll}
		\widetilde{\bd}^T \widetilde{\bQ} \widetilde{\bd} & = \widetilde{\bd}^T \bQ \widetilde{\bd} + \frac{\delta}{\sqrt{n}}\|\widetilde{\bd}\|_2^2\\
		& = \widetilde{\bd}^T \bQ \widetilde{\bd} + \frac{\delta}{\sqrt{n}}\\
		& \geq \bar\bd^T \bQ \bar\bd + \frac{\delta}{\sqrt{n}}\\
		& = \bar\bd^T \bQ \bar\bd + \frac{\delta}{\sqrt{n}}\|\bar\bd\|_2^2,
		\end{array}\]
		where the first and last equalities hold since $\|\bar\bd\|_2 = \|\widetilde{\bd}\|_2=1$, and the inequality holds by definition of $\bar\bd$. Hence, $\widetilde{\bd}$ is not an optimal solution for 
		\[\displaystyle{\min_{\bd\geq \bzero,\;\;\|\bd\|\leq 1}}\bd^T \widetilde{\bQ} \bd,\]
		thus arriving to a contradiction.
	\end{proof}
	
	\vspace{0.2cm}
	
	The result of the above lemma directly implies the following theorem about the hardness of checking	second order stationarity.

	\begin{theorem}\label{approx-SOS-NP}
		For the co-positivity problem
		\begin{equation}
		\min_{\bd \in \mathbb{R}^n} \quad \frac{1}{2} \bd^T \mathbf{Q}\bd \quad \quad \st \quad  \bd \geq \bzero,
		\end{equation}
		there is no  algorithm which can check whether $\bd = 0$ is an $(\epsilon_g, \epsilon_H)$-second order stationary point in polynomial time in $(n,\frac{1}{\epsilon_H})$, unless P = NP.
	\end{theorem}
	\begin{proof}
		The resultt is an immediate consequence of Lemma~\ref{stable-min-equiv}.
	\end{proof}
	
	\begin{remark}
		The result in~\cite{murty1987some} shows that checking for $(\epsilon_g,\epsilon_H)$-second order stationarity is NP-hard in $\left(n, \log(1/\epsilon_H) \right)$. In other words, there is no algorithm which can check whether $\bd$ is an $(\epsilon_g,\epsilon_H)$-second order stationarity point in polynomial time in $(n, \log(1/\varepsilon_H))$, unless P=NP. The result of Theorem~\ref{approx-SOS-NP} is stronger as it shows that checking for $(\epsilon_g,\epsilon_H)$-second order stationarity is NP-hard in $(n, 1/\epsilon_H)$.
	\end{remark}
	
	This negative result shows that we should not expect to have a polynomial-time iterative descent algorithm which can converge to second order stationary points of general convex constrained optimization problems. If such an algorithm exists, one can run that algorithm from the initial point $\bd_0 = \bzero$ and see if it can find a point with negative objective value. This observation shows that in order to have a reasonable descent algorithm (with polynomial per-iteration complexity), we must put the general convex constrained case behind; and develop algorithms for special type of constraints. This transition is the focus of the next section.
	
	\vspace{0.3cm}
	
	\section{Easy Instances of Finding Second Order Stationarity in Constrained Optimization: A Second Order Frank-Wolfe Algorithm}\label{Section_FW}
	
	As discussed in previous sections, although designing polynomial time algorithms for finding second order stationary points is easy when the optimization problem is unconstrained, the same problem becomes very hard in the general convex constrained case. In particular, even for checking second order stationarity, one needs to (approximately) solve a quadratic constrained optimization problem~\eqref{eq:psi_k}, which is NP-hard as shown in  Section~\ref{sec:NPhard}.  However, for some special constraint sets~$\mathcal{F}$, the quadratic constrained optimization problem~\eqref{eq:psi_k} can be solved efficiently. For example, when $\mathcal{F}$ is formed by a fixed number of linear constraints, \cite{hsia2013trust} presents a backtracking approach which can find the solution of~\eqref{eq:psi_k} in polynomial-time. More precisely, by doing an exhaustive backtracking search on the set of constraints, one can find the solution of the problem 
	\begin{equation}\label{eq:subProblemQL}
	\begin{split}
	\min_{\bd}\quad & \bd^T\nabla^2 f(\bx) \bd   \\
	\st \quad & \bx + \bd \in {\cal F},  \, \|\bd\|_2\leq 1\\ &\langle \nabla f(\bx), \bd \rangle \leq 0.
	\end{split}
	\end{equation}
	in polynomial-time when $\mathcal{F} = \{\bx \; \mid \; \mathbf{a}_i^T\bx \leq b_i,\; \textrm{for } i=1,\ldots,m \}$ assuming that $m$ is small and one can afford a search which is exponentially large in $m$.
	
	\vspace{0.2cm}
	
	Assuming that  \eqref{eq:subProblemQL} can be solved efficiently for a given $\mathcal{F}$, a natural question to ask is as follows:
	
	\vspace{0.3cm}
	
	\begin{center}
		\noindent\fbox{
			\parbox{0.9\textwidth}{
				Assume that the constraint set $\mathcal{F}$ is such that the quadratic optimization problem \eqref{eq:subProblemQL} can be solved efficiently. For such a constraint set $\mathcal{F}$, can we find an $(\epsilon_g,\epsilon_H)$-second order stationary point of the general smooth optimization problem~\eqref{eq:General-Cons-Optimization-Prob} efficiently?
			}
		}
	\end{center}
	\vspace{0.4cm}
	
	In this section, we answer this question affirmatively by proposing a polynomial time algorithm for finding an $(\epsilon_g,\epsilon_H)$ second order stationary point of   problem~\eqref{eq:General-Cons-Optimization-Prob} assuming that a quadratic optimization problem of the form~\eqref{eq:subProblemQL} can be solved efficiently at each iteration. The proposed algorithm can be viewed as a simple second order generalization of the Frank-Wolfe algorithm proposed in \cite{lacoste2016convergence}. In particular, in addition to the first order Frank-Wolfe  direction computed by solving~(\ref{eq:X_k}) at $\bx_k$, we also compute a second-order descent direction by solving (\ref{eq:psi_k}) at each iteration. Then we dynamically choose the direction that potentially offers a more significant predicted reduction in the objective value. This dynamic method was used in \cite{curtis2017exploiting} to design an algorithm for unconstrained optimization problems. They show convergence to $(\epsilon_g, \epsilon_H)$-second-order stationary points with complexity ${\cal O}(\max\{\epsilon_g^{-2}, \epsilon_H^{-3})\})$. Our proposed algorithm adapts this method to the constrained scenario while maintaining the same convergence guarantees and complexity bounds.
	
	\vspace{0.2in}

	\textbf{Notations.} Given a sequence of iterates $\{\bx_k\}$ computed by an algorithm for solving (\ref{eq:General-Cons-Optimization-Prob}), we define $\cX_k \triangleq \cX(\bx_k)$ where $\cX(\cdot)$ is defined in~\eqref{eq:X_k}. Moreover, we define 
	\begin{equation}\label{eq:psi_k_alpha}
		{\cal \psi}_{k, \alpha_k} \triangleq -\min_{\bd}\;\; \bd^T \nabla^2f(\bx_k)\bd \quad  \st \quad  \bx_k + \bd \in {\cal F},  \; \|\bd\|\leq 1, \; \langle \nabla f(\bx_k), \bd \rangle \leq \alpha_k,
	\end{equation}
	where $\alpha_k$ is positive scalar dynamically updated by the algorithm.
	
	\vspace{0.1in}
	
	Throughout this section, we make the following standard assumptions.
	\begin{assumption}\label{Assumption1}
		The objective function $f$ is twice continuously differentiable and bounded below by a scalar $f_{min}$ on ${\cal F}$. The constraint set~$\mathcal{F}$ is closed and convex. We assume that functions $\nabla f(\cdot)$ and $\nabla^2 f(\cdot)$ are Lipschitz continuous with Lipschitz constants $L$ and $\rho$, respectively. Furthermore, the gradient sequence $\{\nabla f(\bx_k) \}$ is bounded such that there exists a scalar constant $g_{max} > 0$ such that $\|\nabla f (\bx_k)\|_2 \leq g_{max}$ for all $k \in \mathbb{N}$. Moreover, we assume that the Hessian sequence $\{\nabla^2 f(\bx_k)\}$ is bounded in norm, that is, there exists a scalar constant $H_{max} > 0$ such that $\|\nabla^2 f(\bx_k)\|_2 \leq H_{max}$ for all~$k \in \mathbb{N}$.
	\end{assumption}

	Moreover, we make the following assumption.
	\begin{assumption}\label{Assumption2}
		There exists a constant $\beta > 0$ such that if $\bx \in {\cal F}$ satisfies $\cX(\bx)=0$, then either $\cpsi(\bx) = 0$ or $\cpsi(\bx) \geq \beta$.
	\end{assumption}
	In simpler words, we assume that every first-order stationary point is either a second-order stationary point or has a feasible negative curvature direction $\bd$ satisfying $-\bd^T\nabla^2f(\bx)\bd \geq \beta$. In the unconstrained scheme, this assumption implies that the absolute value of the smallest negative eigenvalue among all strict saddle points is bounded away from zero. If the smallest negative eigenvalue among strict saddle points can be arbitrarily close to zero, then given any prescribed positive constants $\epsilon_g$ and $\epsilon_H$, an approximate strict saddle point can be confused with approximate second-order stationary points. Even in the case of first-order stationarity, if the norm of the gradient of non-critical points can be arbitrarily close to zero, then given any prescribed positive constant $\epsilon$, iterative first-order algorithms can get stuck at these points. Next, we describe our proposed algorithm.

	\subsection{Algorithm Description}
	
	Let $\bx_k$ be the iterate in our algorithm at iteration~$k$. Given point $\bx_k$, we  define the following first order and second order descent directions
	\begin{equation}\label{eq:s_k}
	\begin{split}
	\widehat{\bs}_k  \triangleq  \arg\min_{\bs}\quad &\langle \nabla f(\bx_k), \bs \rangle  \\
	\st \quad & \bx_k + \bs \in {\cal F}, \, \|\bs\|_2\leq 1,
	\end{split}
	\end{equation}
	and
	\begin{equation}\label{eq:d_k}
	\begin{split}
	\widehat{\bd}_k \triangleq \arg\min_{\bd}\quad &\bd^T \nabla^2f(\bx_k)\bd  \\
	\st \quad & \bx_k + \bd \in {\cal F},  \, \|\bd\|_2\leq 1\\ &\langle \nabla f(\bx_k), \bd \rangle \leq \alpha_k,
	\end{split}
	\end{equation}
	where $\alpha_k$ is a positive scalar. Notice that in the unconstrained scenario, $\widehat{\bs}_k = -\nabla f(\bx_k)$ and $\widehat{\bd}_k$ is a feasible negative curvature direction. When $\nabla f(\bx_k) =\bzero$, 
	$\widehat{\bd_{k}}$ become the eigenvector corresponding to the leftmost eigenvalue of the Hessian matrix $\nabla^2 f(\bx_k)$, which lead to the simple directions proposed in \cite{curtis2017exploiting} for the unconstrained scenario. 
	
	\vspace{0.2cm}
	
	The algorithm described below follows a dynamic strategy of choosing between $\widehat{\bs}_k$ and $\widehat{\bd}_k$ for all $k \in \mathbb{N}$. The choice is done based on which direction predicts a larger reduction in the objective. If $\widehat{\bs}_k$ is always chosen, then the algorithm resembles Frank-Wolfe algorithm \cite{lacoste2016convergence}. Hence, our algorithm can be seen as a second order extension of Frank-Wolfe algorithm. Moreover, the algorithm involves dynamic updates for the scalar $\alpha_k$. More specifically, the algorithm updates $\alpha_k$ until the following conditions hold
	\begin{equation}\label{eq:conditions_rho_k}
		\begin{array}{c}
		f(\bx_k) - f\left(\bx_k + \dfrac{2\cpsi_{k, \alpha_k}}{\widetilde{\rho}}\widehat{\bd}_k\right) \geq \dfrac{\cpsi_{k,\alpha_k}^3}{3\widetilde{\rho}^2},\\\\ 
		\widetilde{\rho} \geq 2\cpsi_{k,\alpha},\\\\
		\langle\nabla f(\bx_k) , \widehat{\bd}_k\rangle  \leq \dfrac{\cpsi_{k, \alpha_k}^2}{6\widetilde{\rho}} 
		\end{array}
	\end{equation}
	where $\widetilde{\rho} \triangleq \max\{\rho, 2H_{max}\}$.	The above conditions guarantee feasibility of next iterate and sufficient decrease in the objective value. 	
	
%
%
%
%
%
%
%
%
%
%

	\begin{algorithm}[h]
	\caption{Second Order Frank-Wolfe with Fixed Step-size}\label{alg-FW-cap}
	\textbf{Require:} $\bx_0 \in {\cal F}$, $\widetilde{L} \triangleq \max\{L, g_{max}\}$, $\widetilde{\rho} \triangleq \max\{\rho, 2H_{max}\}$, $\gamma > 1$ and $\alpha_0 \in (0,1]$.
	
	\hrulefill
	
	\begin{algorithmic}[1] \label{alg-FW}
		\Procedure{}{}
		\State Compute $\cX_0$ and $\cpsi_{0,\alpha_0}$ by solving  \eqref{eq:X_k} and \eqref{eq:psi_k_alpha}, respectively.
		\For{$k=0, 1, 2, \ldots$} 
			\State Compute $\widehat{\bs}_k, \, \cX_k$ and $\widehat{\bd}_k, \, \cpsi_{k,\alpha_k}$ using~\eqref{eq:s_k}~and~\eqref{eq:d_k}, respectively.
			\If{$\cpsi_{k,\alpha_k} = \cX_k =0$} 
				\State \textbf{terminate} and \Return $\bx_k$. \label{alg-done}
			\EndIf
		\vspace{0.2cm}
		\Loop \label{alg-loop}
		\If{$\dfrac{\cX_k^2}{2\widetilde{L}} \geq \dfrac{\cpsi_{k,\alpha_k}^3}{3\widetilde{\rho}^2}$} \label{alg:Dynamic_Choice}
			\State set $\bx_{k+1} \gets \bx_k + \dfrac{\cX_k}{\widetilde{L}}\widehat{\bs}_k$ \label{alg-choose-s}
			\State \textbf{exit loop} \label{alg:exit-loop-s}
		\Else
			\If{Conditions~\eqref{eq:conditions_rho_k} hold for $\alpha_k$}
			\vspace{0.2cm}
				\State set $\bx_{k+1} \gets \bx_k + \dfrac{2\cpsi_{k,\alpha}}{\widetilde{\rho}}\widehat{\bd}_k$ \label{alg-choose-d}
				\State \textbf{exit loop} \label{alg:exit-loop-d}
				\Else
				\vspace{0.1cm}
				\State set $\alpha_k \gets \dfrac{1}{\gamma}\alpha_k$ \label{alg-update-alpha}
				\vspace{0.2cm}
				\State Compute $\widehat{\bd}_k, \, \cpsi_{k,\alpha_k}$ using~\eqref{eq:d_k}
				\vspace{0.2cm}
				\State \textbf{go to} Step~\ref{alg:Dynamic_Choice}
			\EndIf
		\EndIf		
		\EndLoop		
	\EndFor
	\EndProcedure

	\end{algorithmic}
\end{algorithm}

	\subsection{Convergence Results}
	
	We first note that regardless of the direction we choose, the step size is either $\dfrac{\cX_k}{\widetilde{L}}$ or $\dfrac{2\cpsi_{k,\alpha_k}}{\widetilde{\rho}}$, which by Cauchy Schwartz and conditions~\eqref{eq:conditions_rho_k} are both less than or equal to $1$. Thus, the iterates generated by the algorithm are always feasible. Also notice that, unlike the algorithms proposed in \cite{lacoste2016convergence, mokhtari2018escaping}, our algorithm does not require any boundedness assumption on the feasible set ${\cal F}$. 
	Before proceeding to the proof of convergence of our algorithm, we first show the following essential Lemma.
	

	\begin{lemma}\label{lm:LowerBound_UpperBound}
		Consider a point $\bbx \in {\cal F}$ satisfying $\cX(\bbx)=0$ and $\cpsi(\bbx) \neq 0$. Then under Assumptions~\ref{Assumption1} and \ref{Assumption2}, for all $\btx \in \ball(\bbx, \delta)$, there exists $\btx^* \in \ball(\btx, 1)\cap {\cal F}$ such that
		\begin{equation}
		-(\btx^* - \btx)^T \nabla^2 f(\btx) (\btx^* - \btx) \geq L(\delta)\label{UpperBound}
		\end{equation}
		and
		\begin{equation}
		\left\langle \nabla f(\btx) , \btx^* - \btx \right\rangle \leq U(\delta), \label{LowerBound}
		\end{equation}
		where
		\[L(\delta) \triangleq \dfrac{\beta}{(1+\delta)^2} - \rho\delta - 5H_{max}\delta^2 - 2H_{max}\delta\]
		and
		\[U(\delta) \triangleq (L+3g_{max})\delta + L \delta^2 .\]
	\end{lemma}
	\begin{proof}
		The proof is relegated to Appendix~\ref{lm:LowerBound_UpperBound-proof}.
	\end{proof}

%
	
	The Lemma above indicates that if $\bx_k \in \ball(\bbx, \delta)$ where $\bbx$ is a strict saddle point, then there exists an $\bar{\alpha} = L(\delta)$ such that $\cpsi(\bx_k , \alpha)$ is bounded below by a positive constant for all $0 <\alpha \leq \bar{\alpha}$.  Hence, if $\bx_k \rightarrow \bar{\bx}$ with $\cX(\bx_k) \rightarrow 0$ and $\cpsi(\bx_k, \alpha) \rightarrow 0$, then $\cX(\bar{\bx})=0$ and $\cpsi(\bar{\bx}, \alpha)=0$ for all $\alpha \leq \bar{\alpha}$. This is formally stated in the Corollary below.	
	
	\begin{corollary}\label{corollary:convergence_asymptotic}
		Consider a sequence of iterates $\{\bx_k\}$ generated by Algorithm~\ref{alg-FW} with $\bx_k \rightarrow \bar{\bx}$. There exists $\bar{\alpha} > 0$ such that if $\cX(\bx_k) \rightarrow 0$ and $\cpsi(\bx_k, \alpha) \rightarrow 0$. Then 
		\[\cX(\bar{\bx}) = 0 \quad \mbox{and} \quad \cpsi(\bar{\bx}, \alpha) = 0,\]
		for all $0 < \alpha < \bar{\alpha}$. 
	\end{corollary}
	
	\begin{proof}
		By the continuity of $\cX(\cdot)$ \cite{conn2000trust}, we get that $\cX(\bar{\bx})=0$. According to~\eqref{LowerBound}, we have \[\left\langle \nabla f(\btx) , \btx^* - \btx \right\rangle \leq U(\delta),\]
		where $U(\delta) \triangleq  (L+3g_{max})\delta + L\delta^2$ is a monotone increasing function for $\delta \in [0,\infty)$ with $U(0)=0$. Moreover,  according to~\eqref{UpperBound}, we have
		\[\left[(\btx^* - \btx)^T \nabla^2 f(\btx) (\btx^* - \btx)\right]^2 \geq L(\delta),\]
		where $L(\delta) \triangleq  \dfrac{\beta}{(1+\delta)^2} -\rho\delta  - 5H_{max}\delta^2 - 2H_{max}\delta$ is a monotone decreasing function for $\delta \geq 0$. Consider $\bar{\delta}$ such that $L(\bar{\delta})=0$. By monotonicity of $U(\delta)$ and $L(\delta)$, given $0 \leq \delta < \bar{\delta}$, there exists $\bx_k^* \in \ball(\bx_k) \cap {\cal F}$ with
		\[\left\langle \nabla f(\bx_k), \bx_k^* - \bx_k \right\rangle \leq \alpha \triangleq  U(\delta)\]
		and
		\[-(\bx_k^* - \bx_k)^T \nabla^2 f(\bx_k) (\bx_k^* - \bx_k) \geq L(\delta).\]
		Hence by definition of $\cpsi_{k, \alpha}$ in~\ref{eq:psi_k_alpha}, $\cpsi_{k, \alpha} \geq L(\delta) > 0$ does not converge to zero which contradicts the assumption that $\cpsi_{k, \alpha} \rightarrow 0$. This completes the proof.
		
	\end{proof}

\begin{theorem}\label{thm:Frank-Wolfe-convergence}
	Under Assumption \ref{Assumption1},
	\[	\lim_{k \rightarrow \infty} \cX_k = \lim_{k \rightarrow \infty} \cpsi_{k, \alpha_k} =0.\]
	In other words, any limit point of the iterates is a second order stationary point.
\end{theorem}

\begin{proof}
	We start by showing that the algorithm exits the loop (Step~\ref{alg-loop}) at every iteration $k$. The claim is obvious when Step~\ref{alg:Dynamic_Choice} tests \textit{True}. We show that the claim holds when Step~\ref{alg:Dynamic_Choice} tests \textit{False}. Assume $\left\langle \nabla f(\bx_k), \widehat{\bd}_k \right\rangle \leq \dfrac{\cpsi_{k, \alpha_k}^2}{6\widetilde{\rho}}$, then using second-order descent lemma \cite{bertsekas1999nonlinear} we obtain
	\begin{flalign}
	f(\bx_{k+1}) &\leq f(\bx_k) + \langle \nabla f(\bx_k), \bx_{k+1}- \bx_k \rangle \nonumber \\
	&+ \dfrac{1}{2}(\bx_{k+1}- \bx)^T \nabla^2 f(\bx_k)(\bx_{k+1}- \bx_k) +  \dfrac{\rho}{6}\|\bx_{k+1}- \bx_k\|_2^3 \nonumber\\
	&\leq f(\bx_k) + \langle \nabla f(\bx_k), \bx_{k+1}- \bx_k \rangle \nonumber \\
	&+ \dfrac{1}{2}(\bx_{k+1}- \bx_k)^T\nabla^2 f(\bx_k) (\bx_{k+1}- \bx_k) +  \dfrac{\widetilde{\rho}}{6}\|\bx_{k+1}- \bx_k\|_2^3 \nonumber \\
	&= f(\bx_k) + \dfrac{2\cpsi_{k,\alpha_k}}{\widetilde{\rho}}\langle \nabla f(\bx_k), \widehat{\bd}_k \rangle + \dfrac{2\cpsi_{k,\alpha_k}^2}{\widetilde{\rho}^2}(\widehat{\bd}_k)^T \nabla^2 f(\bx_k) (\widehat{\bd}_k) + \dfrac{4\cpsi_{k,\alpha_k}^3}{3\widetilde{\rho}^2}\|\widehat{\bd}_k\|_2^3 \nonumber\\
	&\leq f(\bx_k) + \dfrac{\cpsi_{k,\alpha_k}^3}{3\widetilde{\rho}^2} - \dfrac{2\cpsi_{k,\alpha_k}^3}{\widetilde{\rho}^2} + \dfrac{4\cpsi_{k,\alpha_k}^3}{3\widetilde{\rho}^2} \nonumber\\
	&= f(\bx_k) - \dfrac{\cpsi_{k,\alpha_k}^3}{3\widetilde{\rho}^2}\label{reduction-d},
	\end{flalign}
	where the fourth inequality holds since 
	\[\left\langle \nabla f(\bx_k), \widehat{\bd}_k \right\rangle \leq \dfrac{\cpsi_{k, \alpha_k}^2}{6\widetilde{\rho}}, \quad \|\widehat{\bd}_k\| \leq 1, \quad \mbox{and} \quad \cpsi_{k, \alpha_k} = - \widehat{\bd}_k^T\nabla^2 f(\bx_k) \widehat{\bd}_k.\]
	
	Hence, if 
	\begin{equation}\label{rho_alpha_cond}
		\left\langle \nabla f(\bx_k), \widehat{\bd}_k \right\rangle \leq \dfrac{\cpsi_{k, \alpha_k}^2}{6\widetilde{\rho}},
	\end{equation}
	conditions~\eqref{eq:conditions_rho_k} are satisfied. This directly implies that Step~\ref{alg:exit-loop-d} is reached and the algorithm exists the loop. We next show that~\eqref{rho_alpha_cond} will be satisfied at every iteration $k$. Assume the contrary, then 
	\[\left\langle \nabla f(\bx_k), \widehat{\bd}_k \right\rangle > \dfrac{\cpsi_{k, \alpha_k}^2}{6\widetilde{\rho}}.\]
	According to the definition of $\cpsi_{k, \alpha_k}$ in~\eqref{eq:psi_k_alpha}, 
	\begin{equation}\label{eq:Condition-For-Loop-Exit}
	\left( \dfrac{3 \widetilde{\rho}^2 \cX_k^2}{\widetilde{L}} \right)^{2/3} \leq \cpsi_{k, \alpha_k}^2 <  6\widetilde{\rho} \left\langle \nabla f(\bx_k), \widehat{\bd}_k \right\rangle \leq 6\widetilde{\rho}\alpha_k,
	\end{equation}
	where the first inequality holds since Step~\ref{alg:Dynamic_Choice} tested \textit{False}. By the update rule in Steps~\ref{alg-update-alpha}, $\widetilde{\rho}\alpha_k \rightarrow 0 $. Therefore, $\cX_k = 0$ and $\lim_{\alpha_k \rightarrow 0} \cpsi_{k, \alpha_k} =  0$. However, if $\bx_k$ is not a second-order stationary point, $\cpsi_{k, \alpha_k} \geq \cpsi_{k,0} \geq \beta$ which contradicts the convergence of $\widetilde{\rho}\alpha_k$ to zero.\\
	
	We have showed that~\eqref{rho_alpha_cond} will be satisfied for every $k$ unless $\bx_k$ is a second-order stationary point. It follows that Step~\ref{alg:exit-loop-d} or Step~\ref{alg:exit-loop-s} is reached and the algorithm exists the loop. We next show that $\lim_{k \rightarrow \infty} \cpsi_{k, \alpha_k} = 0$. To proceed with the proof, we show the following reduction bound in the objective value		\[f(\bx_{k+1})\leq  f(\bx_k) - \max\left\{\dfrac{\cX_k^2}{2\widetilde{L}}, \, \dfrac{\cpsi_{k, \alpha_k}^3}{3\widetilde{\rho}^2}\right\}.\]
	
	If Step~\ref{alg-choose-s} is reached, then using descent lemma \cite[Appendix~A.24]{bertsekas1999nonlinear}, we obtain
	\begin{flalign}
	f(\bx_{k+1}) &\leq f(\bx_k) + \langle \nabla f(\bx_k), \bx_{k+1}- \bx_k \rangle + \dfrac{L}{2}\|\bx_{k+1}- \bx_k\|_2^2 \nonumber\\
	&\leq f(\bx_k) + \langle \nabla f(\bx_k), \bx_{k+1}- \bx_k \rangle + \dfrac{\tilde{L}}{2}\|\bx_{k+1}- \bx_k\|_2^2 \nonumber\\
	&= f(\bx_k) + \dfrac{\cX_k}{\tilde{L}}\langle \nabla f(\bx_k), \widehat{\bs}_k \rangle + \dfrac{\cX_k^{\,2}}{2\tilde{L}}\| \widehat{\bs}_k\|_2^2\nonumber\\
	&\leq f(\bx_k) - \dfrac{\cX_k^2}{\tilde{L}} + \dfrac{\cX_k^2}{2\tilde{L}}\nonumber\\
	&= f(\bx_k) - \dfrac{\cX_k^{\,2}}{2\tilde{L}}. \label{reduction-s}
	\end{flalign}
	
	Otherwise, Step~\ref{alg-choose-d} is reached and conditions~\eqref{eq:conditions_rho_k} are satisfied. Hence, we get the following reduction bound in the objective value at every iteration
	\begin{equation}\label{reduction}
	f(\bx_{k+1})\leq  f(\bx_k) - \max\left\{\dfrac{\cX_k^2}{2\widetilde{L}}, \, \dfrac{\cpsi_{k, \alpha_k}^3}{3\widetilde{\rho}^2}\right\}.
	\end{equation}
	
	By summing over the iterations, we get
	\begin{equation}\label{eq:SuffDecrease}
	f(\bx_{\ell+1}) - f(\bx_0) = \sum_{k=0}^{\ell}\big(f(\bx_{k+1}) - f(\bx_k)\big) \leq -\sum_{k=0}^{\ell} \,\max\left\{\dfrac{\cX_k^2}{2\widetilde{L}}, \, \dfrac{\cpsi_{k, \alpha_k}^3}{3\widetilde{\rho}^2}\right\}.
	\end{equation}
	
	Hence, since $f$ is bounded below by $f_{min}$, we have
	\begin{flalign*}
	0 \leq \sum_{k=0}^{\ell}  \max\left\{\dfrac{\cX_k^2}{2\widetilde{L}}, \, \dfrac{\cpsi_{k, \alpha_k}^3}{3\widetilde{\rho}^2}\right\} \leq f(\bx_0 ) -f_{min}.
	\end{flalign*}
	
	Thus,
	\[\lim_{k \rightarrow \infty} \cX_k = \lim_{k \rightarrow \infty} \cpsi_{k, \alpha_k} =0.\]
	By Lemma~\ref{corollary:convergence_asymptotic}, we get that every limit point of the iterates is a second order stationary point. This completes the proof.

	\end{proof}

	The next result computes the worst-case complexity required to reach an $\epsilon_g$-first order stationary point and to reach an $(\epsilon_g, \epsilon_H)$-second order stationary point.

\begin{theorem}\label{thm:Frank-Wolfe-complexity}
		Let $\epsilon_g$, $\epsilon_H > 0$.  The number of iterations required for  Algorithm~\ref{alg-FW-cap} to find an $\epsilon_g$-first order stationary point is at most ${\cal O}(\epsilon_g^{-2})$. Moreover, the number of iterations required to find an $(\epsilon_g, \epsilon_H)$-second order stationary point is at most 
		${\cal O}\left(\max\left\{\epsilon_g^{-2}, \epsilon_H^{-3}\right\}\right)$.
\end{theorem}


	\begin{proof}
		Let $r$ be the number of times Step~\ref{alg-update-alpha} is reached. We first show $r = {\cal O}\left( \log\left(max\{ \epsilon_g, \epsilon_H \}\right) \right)$.
		
		Assume that \[\left\langle \nabla f(\bx_k), \widehat{\bd}_k \right\rangle \leq \dfrac{\cpsi_{k, \alpha_k}}{6 \widetilde{\rho}}. \]
		Then, by second-order descent lemma conditions~\eqref{eq:conditions_rho_k} are satisfied which directly implies that Step~\ref{alg:exit-loop-d} is reached and the algorithm exists the loop. Otherwise, if 
		\[\left\langle \nabla f(\bx_k), \widehat{\bd}_k \right\rangle > \dfrac{\cpsi_{k, \alpha_k}}{6 \widetilde{\rho}}. \]
		then by the definition of $\cpsi_{k, \alpha_k}$ we have
		\[\left( \dfrac{3 \widetilde{\rho}^2 \cX_k^2}{\widetilde{L}} \right)^{2/3} \leq \cpsi_{k, \alpha_k}^2 <  6\widetilde{\rho} \left\langle \nabla f(\bx_k), \widehat{\bd}_k \right\rangle \leq 6\widetilde{\rho}\alpha_k,\]
		where the first inequality holds since Step~\ref{alg:Dynamic_Choice} tested \textit{False}.\\
		
		It follows that 
		\[\begin{array}{ll}
		\alpha_0 \gamma^{-r}  &> \max\left\{\dfrac{\cpsi_{k, \alpha_k}^2}{6 \widetilde{\rho}},  \left(\dfrac{9\widetilde{\rho} \cX_k^4}{6^3\widetilde{L}^2}\right)^{1/3}\right\}\\
		&\geq \max\left\{ \dfrac{\epsilon_H^2}{6\widetilde{\rho}}, \left(\dfrac{9\widetilde{\rho} \epsilon_g^4}{6^3\widetilde{L}^2}\right)^{1/3}\right\},
		\end{array}\]
		which implies that
		\begin{equation}\label{eq: bounds_r}
		r \leq \dfrac{1}{\log(\gamma)\alpha_0}\log \left(\max\left\{ \dfrac{\epsilon_H^2}{6 \widetilde{\rho}}, \left(\dfrac{9\widetilde{\rho} \epsilon_g^4}{6^3\widetilde{L}^2}\right)^{1/3}   \right\}\right)={\cal O}\left( \log\left(max\{ \epsilon_g, \epsilon_H \}\right)\right). 
		\end{equation}
		
		Now define the index sets
		\[{\cal G}(\epsilon_g) \triangleq \{ k \,| \, \cX_k >\epsilon_g\} \quad \mbox{and} \quad {\cal H}(\epsilon_H) \triangleq \{ k \,| \, \cpsi_k >\epsilon_H\}.\]
		Next we show bounds on $\big| \, {\cal G}(\epsilon_g)\, \big|$ and $\big| \, {\cal G}(\epsilon_g) \, \cup \, {\cal H}(\epsilon_H) \, \big|$. First notice that the sufficient decrease bound~\eqref{eq:SuffDecrease} implies that
		\begin{equation} \label{eq:ObjDecrease}
		\sum_{k=0}^{\ell} \,\max\left\{\dfrac{\cX_k^2}{2L_k}, \, \dfrac{2\cpsi_{k, \alpha_k}^3}{3\rho_k}\right\} \; \leq \; f(\bx_0) - f_{min},
		\end{equation}
		for every iteration $\ell$. According to the bound \eqref{eq:ObjDecrease}, it is easy to show that the cardinality of the above two sets is bounded by
		\begin{align} \label{eq: bounds_G_H}
		\big| \, {\cal G}(\epsilon_g)\, \big| & \leq \dfrac{2\tilde{L}(f(\bx_0) - f_{min})}{\epsilon_g^2} \nonumber\\
		\big| \, {\cal G}(\epsilon_g) \, \cup \, {\cal H}(\epsilon_H) \, \big| & \leq  \dfrac{f(\bx_0) - f_{min}}{\min\left\{\dfrac{\epsilon_g^2}{2\tilde{L}}, \, \dfrac{2\epsilon_H^3}{3\tilde{\rho}^{\,2}}\right\}}.
		\end{align}
		
		Combining~\eqref{eq: bounds_G_H} and \eqref{eq: bounds_r} we get the desired result.

	\end{proof}

	\newpage
	{
		\bibliography{references}
		\bibliographystyle{abbrv}
	}

	\newpage
	\appendix
	
	\section{Proof of Lemma~\ref{lem:conditions-converge-to-origin}}\label{Apx:lem:StayonLine}
	For a given $k \in \mathbb{N}_{+}$, let $(x_k, y_k) = (x_k, -x_k)$ be the current iterate with $x_k \geq 0$. We first show that iterate $k+1$ generated by projected gradient descent satisfies
	\begin{equation}\label{eq:condition-on-the-line}
	x_{k+1} + y_{k+1}=0.
	\end{equation}
	Then we show that
	\begin{equation}\label{eq:condition-sufficient-decrease}
	x_{k+1} \geq 0, \quad x_k - x_{k+1} \geq 0.
	\end{equation}
	Combining (\ref{eq:condition-on-the-line}) and (\ref{eq:condition-sufficient-decrease}), we will complete our proof.\\
	
	First note that if $x_k =0$, then the result trivially holds. Assume that $x_k >0$, we define
	\begin{flalign*}
	&\bar{x}_{k+1} \triangleq  x_k - \bar{\alpha}\nabla_x f(x_k , -x_k) = x_k - \bar{\alpha}(1-2x_k^2)x_k e^{-2x_k^2} \quad \mbox {and} \\
	&\bar{y}_{k+1} \triangleq  y_k - \bar{\alpha}\nabla_y f(x_k , -x_k) = -x_k + \bar{\alpha}\big[(1-2x_k^2)x_k e^{-2x_k^2} + x_k\big].
	\end{flalign*}
	
	Since $\bar{x}_{k+1} + \bar{y}_{k+1} = \bar{\alpha} x_k > 0$, the point $(\bar{x}_{k+1}, \bar{y}_{k+1})$ is not feasible. By projecting $(\bar{x}_{k+1}, \bar{y}_{k+1})$ to the feasible set $\{(x,y) \, | \, y+x \leq 0\}$, we obtain
	\begin{equation}
	\begin{array}{l}
	x_{k+1} = x_k - \bar{\alpha}(1-2x_k^2)x_k e^{-2x_k^2} -\dfrac{\bar{\alpha}}{2}x_k \quad \mbox{and}\\  y_{k+1} = -x_k + \bar{\alpha}\big[(1-2x_k^2)x_k e^{-2x_k^2} + \dfrac{1}{2}x_k\big].
	\end{array}
	\end{equation}
	Obviously (\ref{eq:condition-on-the-line}) holds. We now show that
	\begin{flalign*}
	&x_{k+1} = x_k \big[1 - \dfrac{\bar{\alpha}}{2} -\bar{\alpha}(1-2x_k^2)e^{-2x_k^2}  \big] \geq 0 \quad \mbox{and}\\
	&x_k - x_{k+1} = x_k\big[\bar{\alpha}(1-2x_k^2)e^{-2x_k^2} + \dfrac{\bar{\alpha}}{2}\big] \geq 0.
	\end{flalign*}
	
	Let $g(x) \triangleq (1-2x^2)e^{-2x^2}$. This function has two global minima $x=\pm 1$, and one global maximum $x=0$. Hence,
	\begin{equation}\label{eq:bound-on-g}
	-e^{-2} \leq g(x) \leq 1, \quad \forall \, \,x.
	\end{equation}
	
	Using (\ref{eq:bound-on-g}), we get 
	\[x_{k+1} = x_k \big[ 1 - \dfrac{\bar{\alpha}}{2} -\bar{\alpha}g(x_k)\big] \geq x_k\big[1 - \dfrac{\bar{\alpha}}{2} -\bar{\alpha}\big] \geq 0,\]
	where the second inequality holds since $\bar{\alpha} < 2/3$ and $x_k \geq 0$. Also,
	\[x_k - x_{k+1} = x_k\bar{\alpha}  \big[g(x_k)  + \dfrac{1}{2}\big] \geq x_k\bar{\alpha}  \big[\dfrac{1}{2}-e^{-2}\big] \geq 0.\]
	
	
	
	Combining (\ref{eq:condition-on-the-line}) and (\ref{eq:condition-sufficient-decrease}), we conclude that for any $\bar{k} \geq k$, $(x_{\bar{k}}, y_{\bar{k}})$ belongs to the compact set $\{(x,y) \,| \, 0 \leq x \leq x_k,\,\, y=-x\} $ which guarantees convergence. Thus
	\[\bar{x} \triangleq \lim_{k \rightarrow \infty} x_{k+1}  = \lim_{k \rightarrow \infty}\Big[ x_k - \bar{\alpha}(1-2x_k^2)x_k e^{-2x_k^2} -\dfrac{\bar{\alpha}}{2}x_k \Big]=\bar{x}\Big[1 - \dfrac{\bar{\alpha}}{2} - \bar{\alpha}g(\bar{x})\Big],   \]
	which implies
	\[\bar{x} =0, \quad \mbox{or} \quad g(\bar{x}) = -\dfrac{1}{2}. \]
	Since $\underset{x}{\max}\,\, g(x) > - e^{-2}$, we get that $\bar{x} =0$ which completes the proof.\hfill$\blacksquare$

	\section{Proof of Theorem~\ref{thm:non-zero-measure-set}} \label{Apx:lem:boxtoline}
	Consider the initial point $(x_0,y_0)$. If we can show that $y_1 = -x_1$ and $x_1 \geq 0$, then using Lemma~\ref{lem:conditions-converge-to-origin}, we conclude that the sequence of iterates $\{(x_k, y_k)\}$ eventually converges to the origin. Thus it suffices to show that there exist an $\epsilon > 0$ such that if
	\begin{equation}\label{eq:intial-condition-to converge-to-origin}
	0.5 -\epsilon \leq  x_0 \leq 0.5, \mbox{ and } -0.5-\epsilon \leq y_0 \leq -0.5,
	\end{equation}
	then the next iterate $(x_1, y_1)$ satisfies 
	\begin{subequations}
		\begin{eqnarray} \label{pert1}
		&x_1 =x_0 +\bar{\alpha}y_0(1-2x_0^2)e^{-x_0^2 - y_0^2} \geq 0, \label{eq: Condition-1-non-zero-measure-set}\\
		&y_1 + x_1 = y_0 + \bar{\alpha}\big[x_0(1-2y_0^2)e^{-x_0^2 -y_0^2} - y_0 \big]  \nonumber\\
		&+ x_0 + \bar{\alpha}y_0(1-2x_0^2)e^{-x_0^2 - y_0^2} \geq 0. \label{eq: Condition-2-non-zero-measure-set}\
		\end{eqnarray}
	\end{subequations}
	The first condition (\ref{eq: Condition-1-non-zero-measure-set}) is satisfied when the step-size $\bar{\alpha}$ is small enough. To prove (\ref{eq: Condition-2-non-zero-measure-set}) we utilize the conditions in (\ref{eq:intial-condition-to converge-to-origin}) to obtain the following inequalities
	\begin{flalign*}
	-2\epsilon \leq & \,  x_0 + y_0 \leq 0,\\
	0.25 + (0.5 - \epsilon)^2 \leq & \, x_0^2 + y_0^2 \leq (0.5 + \epsilon)^2+ 0.25,\\
	0.5 \leq & 1-2x_0^2 \leq \, 1 - 2(0.5 - \epsilon)^2,\\
	1 - 2(0.5 + \epsilon)^2 \leq & \, 1-2y_0^2 \leq 0.5 ,
	\end{flalign*}
	which implies that
	\begin{equation}\label{eq:x_0(1-2y_0^2) + y_0(1-2x_0^2) }
	\begin{array}{ll}
	&x_0(1-2y_0^2) + y_0(1-2x_0^2) \\
	&\geq (0.5 -\epsilon)\big(1 - 2(0.5 + \epsilon)^2\big) -(0.5 +\epsilon)\big(1 - 2(0.5 - \epsilon)^2\big)\\
	&= -3\epsilon + 4\epsilon^3.
	\end{array}
	\end{equation}
	
	Then using (\ref{eq:x_0(1-2y_0^2) + y_0(1-2x_0^2) }), we get
	\[\begin{array}{l}
	x_0 + y_0 - \bar{\alpha}y_0 + \bar{\alpha}\Big[x_0(1-2y_0^2) + y_0(1-2x_0^2)\Big]e^{-x_0^2 -y_0^2}\\
	\geq -2\epsilon + 0.5\bar{\alpha} + \bar{\alpha} \big(-3\epsilon + 4\epsilon^3\big) e^{-0.25}e^{-(0.5 + \epsilon)^2}.
	\end{array}\]
	Note that the right hand side is $0.5\bar{\alpha } + {\cal O}(\epsilon)$ which is greater than or equal to zero for sufficiently small $\epsilon$. This shows that condition (\ref{eq: Condition-2-non-zero-measure-set}) holds, and the proof follows by Lemma~\ref{lem:conditions-converge-to-origin}.\hfill$\blacksquare$

	\section{Proof of Lemma~\ref{lm:LowerBound_UpperBound}}\label{lm:LowerBound_UpperBound-proof}
	\begin{proof}
		Consider $\bbx \in {\cal F}$ satisfying $\cX(\bbx) = 0$ and $\cpsi(\bbx) \neq 0$. By the convexity of ${\cal F}$, we get
		\[\left\langle \nabla f(\bbx) , \bx - \bbx \right\rangle \geq 0 \quad \forall \,\, \bx \in {\cal F},\]
		Moreover, using Assumption~\ref{Assumption2}, there exists $\bx^* \in \ball(\bbx, 1) \, \cap \, {\cal F}$ such that  
		\begin{equation}\label{eq:Existence_Of_x*}
		-(\bx^* - \bbx)^T \nabla f(\bx) (\bx^* - \bbx) \geq \beta \quad \mbox{and} \quad 	\left\langle \nabla f(\bbx), \bx^* - \bbx\right\rangle =0.
		\end{equation}
		
		Let $\btx \in \ball(\bbx, \delta)$. We define 
		\[\btx^* = \begin{cases} \begin{array}{ll} \bx^* & \quad \mbox{if } \|\btx - \bx^*\|_2 \leq 1\\
		\btx + \dfrac{(\bx^* - \btx)}{\|\bx^* - \btx \|_2} & \quad \mbox{otherwise} \end{array}\end{cases}.\]
		We first show~\eqref{UpperBound} by considering the following two cases.\\
		
		$\bullet$ \underline{Case when $\|\btx - \bx^*\|_2 \leq 1$:}\\		
		\begin{equation}\label{eq:UpperBound_X_k_case_1}
		\begin{array}{ll}
		\left\langle \nabla f(\btx) , \btx^* - \btx \right\rangle &= \left\langle \nabla f(\btx) , \btx^* - \bbx \right\rangle + \left\langle \nabla f(\btx) , \bbx - \btx \right\rangle  \\
		\\
		&= \left\langle \nabla f(\bbx) , \bx^* - \bbx \right\rangle + \left\langle \nabla f(\btx) - \nabla f(\bbx) , \bx^* - \bbx \right\rangle \\
		&\quad +  \left\langle \nabla f(\btx) , \bbx - \btx \right\rangle \\
		\\ 
		& \leq L \|\bx^* - \bbx\|_2 \|\btx - \bbx\|_2 + g_{max} \|\bbx - \btx\|_2\\
		\\
		& \leq (L + g_{max}) \delta,\end{array}
		\end{equation}
		where the first inequality holds by~\eqref{eq:Existence_Of_x*}, Cauchy Shwartz, and $L$-Lipschitzness of $\nabla f(\cdot)$. The second inequality holds since $\|\bx^* - \bbx\|_2 \leq 1$ and $\|\btx - \bbx\|_2 \leq \delta$. \\
		\\
		\\
		
		$\bullet$ \underline{Case when $\|\btx - \bx^*\|_2 > 1$:}\\
		\begin{equation}\label{eq:UpperBound_X_k_case_2}
		\begin{array}{ll}
		\left\langle \nabla f(\btx) , \btx^* - \btx \right\rangle &= \left\langle \nabla f(\btx) , \btx^* - \bbx \right\rangle + \left\langle \nabla f(\btx) , \bbx - \btx \right\rangle  \\
		\\
		& = \left\langle \nabla f(\bbx) , \btx - \bbx \right\rangle + \left\langle \nabla f(\bbx) , \dfrac{\bx^* - \btx}{\|\bx^* - \btx\|_2} \right\rangle 
		 + \underbrace{\left\langle \nabla f(\btx) - \nabla f(\bbx) , \btx + \dfrac{\bx^* - \btx}{\|\bx^* - \btx\|_2} - \bbx \right\rangle}_{\leq L\|\btx - \bbx\|_2^2 + L\|\btx - \bbx\|_2}\\
		&\quad +  \underbrace{\left\langle \nabla f(\btx) , \bbx - \btx \right\rangle}_{g_{max}\|\bbx - \btx\|_2}\\
		\\
		& \leq  g_{max} \underbrace{\| \btx - \bbx\|_2}_{\leq \delta} + \underbrace{\left\langle \nabla f(\bbx) , \dfrac{\bx^* - \bbx}{\|\bx^* - \btx\|_2} \right\rangle}_{= 0} + \underbrace{\left\langle \nabla f(\bbx), \dfrac{\bbx - \btx}{\|\bx^* - \btx\|_2} \right\rangle}_{\leq g_{max}\delta} \\
		&\quad + L\underbrace{\|\btx - \bbx\|_2^2}_{\leq \delta^2} + L\underbrace{\|\btx - \bbx\|_2}_{\leq \delta} + g_{max}\|\bbx - \btx\|_2,\\
		\\
		&\leq (L + 3g_{max})\delta + L \delta^2,
		\end{array}
		\end{equation}
		where the second equality holds by our definition of $\btx^*$, the first inequality holds by Cauchy Shwartz and the $L$-Lipschitzness of $\nabla f(\cdot)$, and the second inequality holds by~\eqref{eq:Existence_Of_x*} and since 
		\[1 \leq \|\bx^* - \btx\|_2 \leq \|\bx^* - \bbx\|_2 + \|\bbx - \btx\|_2 \leq 1+\delta.\]
		Here $\|\bx^* - \bbx\|_2 \leq 1$ since $\bx^* \in \ball(\bbx, 1)$ and $\|\bbx - \btx\|_2 \leq \delta$ by our choice of $\btx$. Combining~\eqref{eq:UpperBound_X_k_case_1} and \eqref{eq:UpperBound_X_k_case_2}, we get that
		\[\left\langle \nabla f(\btx) , \btx^* - \btx \right\rangle \leq (L+3g_{max})\delta + L\delta^2.\]
		
		Similarly, we show~\eqref{LowerBound} by considering the following two cases.\\
		
		$\bullet$ \underline{Case when $\|\btx - \bx^*\|_2 \leq 1$:}\\
		\begin{equation}\label{eq:UpperBound_Psi_k_case_1}
		\begin{array}{ll}
		-(\btx^* - \btx)^T \nabla^2 f(\btx)(\btx^* - \btx) & = -(\btx^* - \bbx + \bbx - \btx)^T \nabla^2 f(\btx) (\btx^* - \bbx + \bbx - \btx)\\
		\\
		& = -(\btx^* - \bbx)^T \nabla^2 f(\btx) (\btx^* - \bbx) - 2(\bbx - \btx)^T \nabla^2 f(\btx) (\btx^* - \bbx)  \\
		&\quad- (\bbx - \btx)^T \nabla^2f(\btx) (\bbx - \btx)\\
		\\
		& = -(\bx^* - \bbx)^T \nabla^2 f(\btx) (\bx^* - \bbx) - \underbrace{2(\bbx - \btx)^T \nabla^2 f(\btx) (\bx^* - \bbx)}_{\geq -2H_{max}\|\bx^* - \bbx\|\delta}  \\
		&\quad- \underbrace{(\bbx - \btx)^T \nabla^2f(\btx) (\bbx - \btx)}_{\geq -H_{max}\delta^2}\\
		\\		
		& \geq \underbrace{-(\bx^* - \bbx)^T \nabla^2 f(\bbx) (\bx^* - \bbx)}_{\beta} -2(\bx^* - \bbx)^T\left(\nabla^2 f(\btx) - \nabla^2 f(\bbx) \right)  (\bx^* - \bbx)   \\
		&\quad - 2H_{max}\|\bx^* - \bbx\|_2\delta -H_{max}\delta^2 \\
		\\
		&\geq \beta - \rho\delta - 2H_{max}\delta - H_{max}\delta^2
		,\end{array}
		\end{equation}
		where the first inequality holds by Cauchy Shwartz and since $\|\btx - \bbx\|_2 \leq \delta$. The second inequality holds by Assumption~\ref{Assumption2}, $\rho$-Lipschitzness of $\nabla^2 f(\cdot)$, and since $\|\bx^* - \bbx\|_2 \leq 1$. 
		\newpage
		
		$\bullet$ \underline{Case when $\|\btx - \bx^*\|_2 > 1$:}\\
		\begin{equation}\label{eq:UpperBound_Psi_k_case_2}
		\begin{array}{ll}
		-(\btx^* - \btx)^T \nabla^2 f(\btx)(\btx^* - \btx) & = -(\btx^* - \bbx + \bbx - \btx)^T \nabla^2 f(\btx) (\btx^* - \bbx + \bbx - \btx)\\
		\\
		& = -(\btx^* - \bbx)^T \nabla^2 f(\btx) (\btx^* - \bbx) - 2(\bbx - \btx)^T \nabla^2 f(\btx) (\btx^* - \bbx)  \\
		&\quad- (\bbx - \btx)^T \nabla^2f(\btx) (\bbx - \btx)\\
		\\
		& = -\left(\btx - \bbx + \dfrac{\bx^* - \btx}{\|\bx^* - \btx\|_2}\right)^T \nabla^2 f(\btx) \left(\btx - \bbx + \dfrac{\bx^* - \btx}{\|\bx^* - \btx\|_2}\right) \\
		&\quad - 2(\bbx - \btx)^T \nabla^2 f(\btx) \left(\btx - \bbx + \dfrac{\bx^* - \btx}{\|\bx^* - \btx\|_2}\right)  \\
		&\quad- (\bbx - \btx)^T \nabla^2f(\btx) (\bbx - \btx)\\
		\\		
		& \geq \underbrace{-\left(\btx - \bbx\right)^T \nabla^2 f(\btx) \left(\btx - \bbx \right)}_{\geq -H_{max}\delta^2} - \left(\dfrac{\bx^* - \btx}{\|\bx^* - \btx\|_2}\right)^T \nabla^2 f(\btx) \left(\dfrac{\bx^* - \btx}{\|\bx^* - \btx\|_2}\right)\\
		&\quad - 2(\btx - \bbx)^T \nabla^2 f(\btx) \left(\dfrac{\bx^* - \btx}{\|\bx^* - \btx\|_2}\right)  -2(\bbx - \btx)^T \nabla^2 f(\btx) (\btx - \bbx)\\
		&\quad  + 2(\btx - \bbx)^T \nabla^2 f(\btx) \dfrac{(\bx^* - \btx)}{\|\bx^* - \btx\|_2} - H_{max}\delta^2\\
		\\
		& \geq -4H_{max}\delta^2  - \left(\dfrac{\bx^* - \btx}{\|\bx^* - \btx\|_2}\right)^T \nabla^2 f(\bbx) \left(\dfrac{\bx^* - \btx}{\|\bx^* - \btx\|_2}\right)\\
		&\quad - \left(\dfrac{\bx^* - \btx}{\|\bx^* - \btx\|_2}\right)^T \left( \nabla^2 f(\btx) - \nabla^2 f(\bbx) \right)  \left(\dfrac{\bx^* - \btx}{\|\bx^* - \btx\|_2}\right)\\
		\\
		& \geq -4H_{max}\delta^2  - \rho \delta - \underbrace{\left(\dfrac{\bx^* - \bbx}{\|\bx^* - \btx\|_2}\right)^T \nabla^2 f(\bbx) \left(\dfrac{\bx^* - \bbx}{\|\bx^* - \btx\|_2}\right)}_{\geq \beta/(1+\delta)^2}\\
		&\quad - \left(\dfrac{\bbx - \btx}{\|\bx^* - \btx\|_2}\right)^T \nabla^2 f(\bbx)  \left(\dfrac{\bbx - \btx}{\|\bx^* - \btx\|_2}\right) \\
		&\quad - 2\left(\dfrac{\bx^* - \bbx}{\|\bx^* - \btx\|_2}\right)^T \nabla^2 f(\bbx)  \left(\dfrac{\bbx - \btx}{\|\bx^* - \btx\|_2}\right)\\
		\\
		&\geq \dfrac{\beta}{(1+\delta)^2} -\rho\delta  - 5H_{max}\delta^2 - 2H_{max}\delta\\
		,\end{array}
		\end{equation}
		where the third equality holds by our definition of $\btx^*$, and the first and second inequalities hold by Cauchy Shwartz, Assumption~\ref{Assumption1} and since $\|\bbx - \btx\|_2 \leq \delta$. Here the third and fourth inequalities hold by $\rho$-Lipschitzness of $\nabla^2 f(\cdot)$, Cauchy Shwartz and Assumption~\ref{Assumption2}. Combining~\eqref{eq:UpperBound_Psi_k_case_1} and \eqref{eq:UpperBound_Psi_k_case_2}, we get that
		\[	-(\btx^* - \btx)^T \nabla^2 f(\btx)(\btx^* - \btx) \geq \dfrac{\beta}{(1+\delta)^2} -\rho\delta  - 5H_{max}\delta^2 - 2H_{max}\delta.\]
		
	\end{proof}	

\end{document}